\documentclass[11pt,twoside]{article}
\usepackage{amsthm}
\usepackage{amsmath}
\usepackage{amssymb}
\usepackage{array}
\usepackage{amsfonts,amscd}
\usepackage{exscale}
\usepackage{eucal}
\usepackage{latexsym,amsmath}
\usepackage{indentfirst}
\newcommand{\e}{\varepsilon}

\numberwithin{equation}{section}
\newtheorem{Prop}{\bf Proposition}[section]
\newtheorem{Cor}{\bf Corollary}[section]
\newtheorem{defn}{\bf Definition}[section]

\newtheorem{Ex}{\bf Example}[section]
\newtheorem{Th}{Theorem}[section]

\begin{document}
\def \b{\Box}

\begin{center}
{\Large {\bf ALGEBRAIC PROPERTIES OF $~{\cal G}-$GROUPOIDS}}
\end{center}

\begin{center}
{\bf Mihai IVAN}
\end{center}

\setcounter{page}{1}

\pagestyle{myheadings}

{\small {\bf Abstract}. The main purpose of  this paper is to give
a new definition for the notion of group-groupoid. Also, several
basic properties of group-groupoids are established.}
{}

\section{Introduction}
\indent\indent There are basically two ways of approaching
groupoids. The first one is the category theoretical approach. The
second one is algebraically considering them as a particular
generalization of the structure of group. Groupoids are like
groups, but with partial multiplication; i.e., only specified
pairs of elements can be multiplied and inverses with respect to
the multiplication exist for each element.

The groupoid was introduced by H. Brandt [Math. Ann.,
\textbf{96}(1926), 360-366] and it is developed by P. J. Higgins
in \cite{higgi}.

The notion of group-groupoid was defined by R. Brown and Spencer
in the paper \cite{brspen}. A group-groupoid is viewed as a
groupoid object in the category of groups (\cite{brown}).

 The groupoids, group-groupoids and their generalizations (topological
groupoids and topological group-groupoids, Lie groupoids and Lie
group-groupoids etc.) are mathematical structures that have proved
to be useful in many areas of science (see for instance
\cite{brown06, wein, mack, icozgu, guicoz, icguoz, mcuk}).

 The paper is organized as follows. In Section 2  we present some basic facts
 about  groupoids. In Section 3  we present the notion of group-groupoid as given in \cite{brmuc} and
 \cite{brspen}. We prove a main theorem for
characterization the group groupoids. This is used for give a new
definition for the concept of group-groupoid. Applying these facts
we establish some important properties in the category of
group-groupoids.

\section{Preliminaries about groupoids}
\indent\indent We recall the minimal necessary backgrounds on
groupoids (\cite{mack, wein}).

A way to think of a groupoid is to say that a {\it groupoid} is a
small category in which every morphism is an isomorphism
(\cite{higgi, brown}). Thus, a groupoid consists of:\\
$~~~\bullet~$ two sets $G$ and $G_{0}$ called the {\it set of
arrows} ({\it or {\it elements}}) and a {\it set of objects} (or
the {\it base}) respectively, together with maps $\alpha, \beta: G
\to G_{0},~\e:G_{0}\to G $ such that $\alpha\circ \e =\beta\circ
\e =Id_{G_{0}} $;\\
$~~~\bullet~$ if $x,y\in G$ and $\beta(x)=\alpha(y), $ then a {\it
product} $x y$ exists such that $\alpha(x y)=\alpha(x)$ and
$\beta(x y)=\beta(y)$, and this product is associative;\\
$~~~\bullet~ \e(u)\in G$ (denoted  with
$1_{u}$) for $u\in G_{0}$, act as identities, and\\
$~~~\bullet~$ each $x\in G$ has an inverse $x^{-1}\in G$
 with $\alpha(x^{-1})=\beta(x), \beta(x^{-1})=\alpha(x), x x^{-1} = 1_{\alpha(x)} $ and $ x^{-1} x = 1_{\beta(x)}.$

In the following definition we describe the groupoid as algebraic
structure.

\begin{defn}{\rm (\cite{mack})~ A \textit{groupoid $ G $
over} $ G_{0} $  is a pair $ (G, G_{0})$ of sets endowed with two
surjective maps $\alpha,\beta: G \to G_{0}$ ({\it source} and {\it
target}), a partially {\it multiplication} $ m:G_{(2)}:=\{(
x,y)\in G\times G | \beta(x)=\alpha(y)\}\to G,~(x,y)\mapsto
m(x,y):=x\cdot y, $ ($G _{(2)}$ is the {\it set of composable
pairs}), an injective map $ \e:~G_{0} \to G $ ({\it inclusion
map}) and a map $ i :G \to G,~x\mapsto  i(x):=x^{-1}$
(\textit{inversion}), which verify the following conditions:

(G1)~({\it associativity}): if  $(x,y)\in G_{(2)}$ and $(y,z)\in
G_{(2)}$, then so $(x\cdot y,z)\in G_{(2)}$ and $(x,y\cdot z)\in
G_{(2)}$, and the relation,  $~(x\cdot y)\cdot z=x\cdot(y\cdot z)$
is satisfied;

(G2)~({\it units}): for each  $ x\in G $ it follows $
(\e(\alpha(x)),x)\in G_{(2)}, (x,\e(\beta(x)))\in  G_{(2)} $ and
$~\e(\alpha(x))\cdot x = x = x\cdot \e(\beta(x))$;

 (G3)~({\it inverses}): for each $ x\in G $ it follows $
(x^{-1},x)\in G_{(2)},~ (x,x^{-1}) \in G_{(2)}~$ and $
~x^{-1}\cdot x = \e(\beta(x)),~ x\cdot
x^{-1}=\e(\alpha(x)).$}\hfill $\Box$
\end{defn}

\markboth{Mihai Ivan}{Algebraic properties  of $~{\cal
G-}$groupoids}

We write sometimes $ x y $ for $ m(x,y) $, if $ (x,y) \in
G_{(2)}.$ Whenever we write a product in a given groupoid, we are
assuming that it is defined.

The element $\varepsilon(\alpha(x))$ (resp.,
$\varepsilon(\beta(x))$) is called the \textit{left}
 (resp., \textit{right unit}) of $x;$ $~\varepsilon(G_{0})$ is
called the \textit{unit set};  $ x^{-1}$ is called the
\emph{inverse} of $x$.

 For a groupoid we use the notation $ (G, \alpha, \beta, m, \e, i, G_{0}) $  or $ (G, G_{0});~ \alpha,
\beta, m, \e, i $ are called the {\it structure functions} of $G$.
For each $ u\in G_{0} $, the set $ \alpha^{-1}(u)$ (resp., $
\beta^{-1}(u)$) is called \textit{$\alpha-$fibre} (resp.,
\textit{$\beta-$fibre}) of $ G$ at $ u$. The map $ ( \alpha,
\beta): G \to G_{0}\times G_{0} $ defined by $ (\alpha ,
\beta)(x):= (\alpha(x), \beta(x)),~(\forall)~ x\in G $ is called
the {\it anchor map} of $ G.$ A groupoid is {\it transitive}, if
its anchor map is surjective.

For any $ u\in G_{0},$ the set $ G(u):= \alpha^{-1}(u)\cap
\beta^{-1}(u)$ is a group under the restriction of the
multiplication, called the {\it isotropy group at} $ u $ of the
groupoid $ (G, G_{0}). $

If $(G, \alpha, \beta, m, \varepsilon, i, G_{0} )$ is a groupoid
such that  $G_{0}\subseteq G$ and $\varepsilon:G_{0}\to G$ is the
inclusion, then we say that $(G, \alpha, \beta, m, i, G_{0} )$ is
a {\it $G_{0}-$groupoid} or a {\it Brandt groupoid}.

In the following proposition we summarize some basic properties of
groupoids obtained directly from definitions.

\begin{Prop}
{\rm (\cite{ivan})} If $ ( G, \alpha, \beta, m,\e, i, G_{0}) $ is
a groupoid, then:

$(i)~~~\alpha(x y) = \alpha(x) $ and $~\beta(x y)=\beta(y) $ for
any $ (x,y)\in G_{(2)};$

$(ii)~~\alpha( x^{-1}) =\beta (x)~~\hbox{and}~~\beta(x^{-1})
=\alpha(x),~ (\forall) x\in G;$

$(iii)~\alpha(\varepsilon(u))=u,~ \beta (\varepsilon(u))=u ,~
\varepsilon(u)\cdot\varepsilon(u)=\varepsilon(u),~
(\varepsilon(u))^{-1}=\varepsilon(u), ~~(\forall) u\in G_{0};$

$(iv)~~~$ if $(x,y)\in G_{(2)},$ then $(y^{-1},x^{-1})\in G_{(2)}$
and  $~(x\cdot y)^{-1}=y^{-1}\cdot x^{-1};$

$(v)~~\varphi : G(\alpha(x)) \to G(\beta(x)),~\varphi (z):= x^{-1}
z x $ is an isomorphism of groups.

$(vi)~$ if $ (G, G_{0} ) $ is transitive, then all isotropy groups
are isomorphic.\hfill$\Box$
 \end{Prop}

Applying Proposition 2.1, it is easily to prove the following
proposition.

\begin{Prop}
{\rm (\cite{ivan})} The structure functions of a groupoid $ ( G,
\alpha, \beta, m, \varepsilon, i, G_{0}) $ verifies the following
relations:

$\alpha \circ i = \beta,~~ \beta \circ i = \alpha, ~~i\circ
\varepsilon = \varepsilon ~~ \alpha \circ \varepsilon= \beta \circ
\varepsilon= Id_{G_{0}}~~ \hbox{and}~~ i \circ i = Id_{G}.$
\hfill$\Box$
\end{Prop}

A subgroupoid of the groupoid $ (G, G_{0}) $ is a subcategory $
(H, H_{0}) $ of $ (G, G_{0}) $ such that $ (H, H_{0}) $ is itself
a groupoid.  A morphism between two groupoids is essentially  a
functor. These concepts are described in the following definition.

\begin{defn}
{\rm (\cite{ivan99})  Let $(G ,\alpha, \beta, m, \e, i, G_0)$  be
a groupoid. A pair of nonempty subsets $(H,H_0)$ where $H\subseteq
G $ and ${H_0}\subseteq G_0 $, is called \textit{subgroupoid} of
$G$, if:

$(1)~~\alpha (H)=H_0\quad \hbox{and}\quad \beta (H)=H_0;$

$(2)~~H $ is closed under partially multiplication and inversion,
that is:\\[0.1cm]
$(i)~\forall~x,y\in H$  such that $(x,y)\in G_{(2)}
\Longrightarrow ~x\cdot y\in H$; $~(ii)~\forall~x\in
H\Longrightarrow x^{-1}\in H.$}\hfill$\Box$
\end{defn}
\begin{Ex}
{\rm $(i)~$ A nonempty set $~G_{0}~$ is  a groupoid over $~G_{0},
$ called the {\it null groupoid}. For this, we take $~\alpha =
\beta = \varepsilon = i = Id_{G_{0}}~$ and $~ u\cdot u = u~$ for
all $~u\in G_{0}.~$

$(ii)~$ A group $ G $ having $ e $ as unity is  a $ \{ e
\}-$groupoid with respect to structure functions: $~\alpha
(x)=\beta (x):= e$; $\varepsilon:\{e\} \to G,~\varepsilon(e):=
e,~G_{(2)}= G\times G,~ m (x,y):= xy~$ and $~i:G \to G, ~i(x):=
x^{-1}$. Conversely, a groupoid with one unit is a group.

$(iii)~$ The Cartesian product $~G:= X \times X~$ has a structure
of groupoid over  $ X $  by taking the structure functions as
follows: $~\overline{\alpha}(x,y):= x,~ \overline{\beta}(x,y):=
y;~$ the elements $~ (x,y)~$ and $~(y^{\prime},z)~$ are composable
in $~G:=X \times X~$ iff $~y^{\prime} = y~$ and we define
$~(x,y)\cdot (y,z) = (x,z)~$, the inclusion map $
\overline{\varepsilon}:X \to X\times X$ is given by
$\overline{\varepsilon}(x):=(x,x)$ and the inverse of $~(x,y)~$ is
defined by $~(x,y)^{-1}:= (y,x).~$ This is called the {\it pair
groupoid} associated to set $X$. Its unit set is $~\e(X)=\{
(x,x)\in X\times X | x\in X\}.$  The isotropy group $G(x)$ at
$x\in X$ is the null group $\{(x,x)\}$.

$(iv)~$ For the groupoids $ (G, \alpha_{G}, \beta_{G},
m_{G},\e_{G}, i_{G}, G_{0}) $ and $ (K, \alpha_{K}, \beta_{K},
m_{K}, \e_{K}, i_{K}, K_{0}), $ one construct the groupoid $
(G\times K, G_{0}\times K_{0})$ with the structure functions given by:\\
$~\alpha_{G\times K}(g,k) = (\alpha_{G}(g),\alpha_{K}(k));~
~~\beta_{G\times K}(g,k) = (\beta_{G}(g),\beta_{K}(k));~$\\
$m_{G\times K}((g,k),(g^{\prime},k^{\prime}))= (
m_{G}(g,g^{\prime}), m_{K}(k,
k^{\prime})),~(\forall)~(g,g^{\prime})\in G_{(2)},~(k,k^{\prime})\in K_{(2)};$\\
$~\e_{G\times K}(u,v) = (\e_{G}(u),\e_{K}(v)),~(\forall)~u\in
G_{0}, v\in K_{0}~$ and $~i_{G\times K}(g,k) = ( i_{G}(g),
i_{K}(k))$.

This groupoid is called the {\it direct product of $ (G, G_{0}) $
and $ (K, K_{0})$}}.\hfill$\b$
\end{Ex}
\begin{defn}
{\rm (\cite{mack})
 A {\it morphism of groupoids} or {\it groupoid morphism} from
$ (G, G_{0}) $ into $ (G^{\prime}, G_{0}^{\prime}) $ is a pair
$(f, f_{0}), $  where $ f:G \to G^{\prime} $ and $ f_{0}: G_{0}
\to G_{0}^{\prime} $ such that the following conditions hold:

$(1)~~~\alpha^{\prime}\circ f = f_{0} \circ
\alpha,~\beta^{\prime}\circ f = f_{0} \circ \beta $;

$(2)~~~f(m(x,y)) = m^{\prime}(f(x),f(y))~$ for all $~(x,y)\in
G_{(2)}.$}\hfill$\Box$
\end{defn}

A groupoid morphism $ (f, Id_{G_{0}}): (G, G_{0}) \to (G^{\prime},
G_{0})$ is called {\it $ G_{0}-$morphism of groupoids}. A groupoid
morphism $ (f, f_{0}): (G, G_{0}) \to (G^{\prime},
G_{0}^{\prime})$ such that $f$ and $f_{0}$ are bijective maps, is
called {\it isomorphism of groupoids}.

\begin{Prop}{\rm (\cite{ivan})}
If $ (f, f_{0}):(G, G_{0}) \to (G^{\prime}, G_{0}^{\prime})$ is a
 groupoid morphism, then:\\[-0.4cm]
\[
f \circ \varepsilon = \varepsilon^{\prime}  \circ
f_{0}~~~\hbox{and}~~~ f \circ i = i^{\prime}  \circ
f.~~~\hfill\Box\\[-0.2cm]
\]
\end{Prop}

For more details on groupoids and their applications, see
\cite{brown, mack, rare, ivan99}. Special properties of some
classes of Brandt groupoids are presented in \cite{ivan02} and
\cite{vpop}.

\section{Basic properties in category of group-groupoids}

 In this section we refer to notion of group-groupoid \cite{brspen}.

A group structure on a nonempty set is regarded as an universal
algebra determined by a binary operation, a nullary operation
 and an unary operation.

Let $ (G, \alpha, \beta, m, \varepsilon, i, G_{0})$ be a groupoid.
We suppose that on $G$ is defined a group structure $\omega:
G\times G \to G,~(x,y) \to \omega (x,y):=x\oplus y$. Also, we
suppose that on $G_{0}$ is defined a group structure $\omega_{0}:
G_{0}\times G_{0} \to G_{0},~(u,v) \to \omega_{0}(u,v):=u\oplus
v$. The unit element of $G$ (resp., $G_{0}$) is $e$ (resp.,
$e_{0}$); that is $\nu : \{\lambda\} \to G,~\lambda \to
\nu(\lambda):= e $ (resp., $\nu_{0} :\{\lambda\} \to
G_{0},~\lambda \to \nu_{0}(\lambda):= e_{0}$) (here $\{\lambda\}$
is a singleton). The inverse of $x\in G$ (resp., $u\in G_{0}$)  is
denoted by $\bar{x}$ (resp., $\bar{u}$ ); that is $\sigma :G \to
G,~x \to \sigma (x):=\bar{x}$ (resp., $\sigma_{0}:G_{0} \to
G_{0},~u \to \sigma_{0}(u):= \bar{u}$).

\begin{defn}
{\rm (\cite{brspen})  A {\it group-groupoid} or {\it $~{\cal
G}-$groupoid}, is a groupoid $ (G, G_{0})$ such that the following
conditions hold:

$(i)~~ (G, \omega, \nu, \sigma)$ and $ (G_{0}, \omega_{0},
\nu_{0}, \sigma_{0})$ are groups.

$(ii)~~$ The maps $~(\omega, \omega_{0}):(G\times G, G_{0}\times
G_{0}) \to (G,
  G_{0}),~(\nu, \nu_{0}):(\{\lambda \}, \{\lambda \}) \to (G, G_{0}) $ and $ (\sigma, \sigma_{0}): (G, G_{0}) \to (G,
  G_{0}) $ are groupoid morphisms.}\hfill$\Box$
\end{defn}

We shall denote a group-groupoid by $ (G, \alpha, \beta, m, \e, i,
\oplus, G_{0})$.

\begin{Prop}
If  $~(G, \alpha, \beta, m, \e, i, \oplus,  G_{0})$ is a
group-groupoid, then:

$(i)~ $ the multiplication $m$ and binary operation $\omega$ are
compatible, that is:\\[-0.3cm]
\begin{equation}
(x\cdot y)\oplus (z\cdot t)= (x\oplus z)\cdot (y\oplus
t),~~~(\forall) (x,y), (z,t)\in G_{(2)};\label{(3.1)}
\end{equation}

$(ii)~ \alpha, \beta: (G, \oplus) \to (G_{0}, \oplus),~ i: (G,
\oplus) \to (G, \oplus)~$  and $~\e: (G_{0}, \oplus)\to (G,
\oplus) $ are morphisms of groups; i.e., for all $ x,y\in G $ and
$u,v\in G_{0},$ we have:\\[-0.3cm]
\begin{equation}
\alpha(x\oplus y)= \alpha(x)\oplus \alpha(y),~~~\beta(x\oplus y)=
\beta(x)\oplus \beta(y),~~~i(x\oplus y)= i(x)\oplus
i(y);\label{(3.2)}
\end{equation}
\begin{equation}
\e(u\oplus v)= \e(u)\oplus \e(v);\label{(3.3)}
\end{equation}

$(iii)~$ the multiplication $m$ and the unary operation $\sigma$
are compatible, that is:\\[-0.3cm]
\begin{equation}
\sigma(x\cdot y)= \sigma(x)\cdot \sigma(y),~~~(\forall) (x,y)\in
G_{(2)}.\label{(3.4)}
\end{equation}
\end{Prop}
\begin{proof}
By Definition 2.3, since $ (\omega, \omega_{0})$ is a groupoid morphism it follows that:\\[0.1cm]
$(1)~~~\alpha\circ \omega = \omega_{0}\circ (\alpha\times
\alpha)~$ and $~\beta\circ \omega = \omega_{0}\circ (\beta\times
\beta);~$\\[0.1cm]
$(2)~~~\omega(m_{G\times G}((x,y),(z,t)))=m_{G}(\omega(x,z),
\omega(y,t)),~ (\forall)~(x,y), (z,t)\in G_{(2)}.$

$(i)~$ We have\\[0.1cm]
$\omega(m_{G\times G}((x,y),(z,t)))=\omega(m_{G}(x,y),m_{G}(z,t))=
\omega(x\cdot y, z\cdot t) = (x\cdot y)\oplus (z\cdot t)~$
and\\[0.1cm]
$m_{G}(\omega(x,z), \omega(y,t))= m_{G}(x\oplus z, y\oplus t)=
(x\oplus z)\cdot (y\oplus t).$

Using $(2)$ one obtains $ (x\cdot y)\oplus (z\cdot t) = (x\oplus
z)\cdot (y\oplus t),$ and $(3.1)$ holds.

$(ii)~$ For each $(x,y)\in G\times G$, we have\\[0.1cm]
$\alpha (\omega(x,y)) =\alpha (x\oplus y)~$ and $~\omega_{0}(
(\alpha\times \alpha)(x,y))=\omega_{0}(\alpha(x), \alpha (y))=
\alpha(x)\oplus \alpha (y).$

According to the first equality $(1)$, it follows $ \alpha
(x\oplus y)=\alpha(x)\oplus \alpha (y),$ and the first relation of
$(3.2)$ holds. Similarly, we prove that the second relation of
$(3.2)$ holds.

Since $(\omega, \omega_{0})$ is a groupoid morphism, by Proposition 2.3, it follows\\[0.1cm]
$(3)~~~\omega\circ (\e\times \e)= \e \circ \omega_{0}~$ and $~
i\circ \omega = \omega\circ (i\times i).$

For each $(x,y)\in G\times G$, we have\\[0.1cm]
$i (\omega (x,y)) = i(x\oplus y)~$ and $~\omega ((i\times i)(x,y))
= \omega (i(x), i(y)) = i(x)\oplus i(y).$

Using now the second  equality $(3)$, it follows $ i(x\oplus y)=
i(x)\oplus i(y),$ and the third relation $(3.2)$ holds.
 For each $(u,v)\in G_{0}\times G_{0}$, we have\\[0.1cm]
$\omega((\e\times \e)(u,v))= \omega(\e(u),\e(v))= \e(u)\oplus
\e(v)~$ and $~\e(\omega_{0}(u,v))= \e(u\oplus v).$

 From the first equality $(3)$, it follows $\e(u\oplus v)= \e(u)\oplus
 \e(v).$ Hence, $(3.3)$ holds.

$(iii)~$ Since $~(\sigma, \sigma_{0}) $ is a groupoid morphism,
for all $(x,y)\in G_{(2)}$ we have\\[0.1cm]
$\sigma(m(x,y))=m(\sigma(x), \sigma(y))$; i.e., $\sigma(x\cdot
y)=\sigma(x)\cdot \sigma(y).$
 Hence  $(3.4)$ holds.
\end{proof}

The relation $(3.1)$ (resp., $(3.4)$) is called the {\it
interchange law} between groupoid multiplication $m$ and group
operation $\omega$ (resp., $\sigma$).

 From Proposition 3.1 follows the following corollary.
\begin{Cor}
Let $ (G, \alpha, \beta, m, \e, i, \oplus,  G_{0})$ be a ${\cal
G}-$groupoid. Then:

$(i)~~~$ The source and target $ \alpha, \beta : G \to G_{0} $ are
group epimorphisms, and\\[-0.3cm]
\begin{equation}
\alpha(e)=\beta(e)=e_{0},~~~~~\alpha(\bar{x})=\overline{\alpha(x)}~~\hbox{and}~~
\beta(\bar{x})=\overline{\beta(x)},~~ (\forall)~x\in
G;\label{(3.5)}
\end{equation}

$(ii)~~~$ The inclusion map $ \varepsilon: G_{0} \to G $ is a
group monomorphism, and\\[-0.3cm]
\begin{equation}
\varepsilon(e_{0})=e,~~~\e(\bar{u})=\overline{\e(u)},~(\forall)~u\in
G_{0};\label{(3.6)}
\end{equation}

 $(iii)~~$ The inversion $ i : G \to G $ is a group
automorphism, and\\[-0.3cm]
\begin{equation}
i(e)=e,~~~ i(\bar{x})=\overline{i(x)},~(\forall)~x\in G;
\label{(3.7)}
\end{equation}

$(iv)~~$ For all $x, y\in G$, we have\\[-0.3cm]
\begin{equation}
\sigma(x\oplus y)= \sigma(y)\oplus \sigma(x)~~~\hbox{and}~~~
\sigma(\sigma ({x}))=x.~~~\hfill\Box \label{(3.8)}
\end{equation}
\end{Cor}
We say that the group-groupoid $ (G, \alpha, \beta, m, \e, i,
\oplus, G_{0})$ is a {\it commutative group-groupoid}, if the
groups $G$ and $G_{0}$ are commutative.
\begin{Cor}
Let $ (G, \alpha, \beta, m, \e, i, \oplus,  G_{0})$ be a
commutative ${\cal G}-$groupoid. Then:\\[-0.3cm]
\[
\overline{x\oplus y}= \bar{x} \oplus \bar{y}, ~(\forall)~ x,y \in
G.
\]
\end{Cor}
\begin{proof}
It is an immediate consequence of the relation $(3.8).$
\end{proof}
\begin{Prop}
If $~(G, \alpha, \beta, m, \e, i, \oplus,  G_{0})$ is a ${\cal
G}-$groupoid, then:\\[-0.3cm]
\begin{equation}
e\cdot y = y,~~~(\forall) y\in
\alpha^{-1}(e_{0})~~~\hbox{and}~~~x\cdot e = x,~~~(\forall) x\in
\beta^{-1}(e_{0});\label{(3.9)}
\end{equation}
\begin{equation}
x\cdot (y\oplus t) = x\cdot y \oplus t,~~~(\forall) (x,y)\in
G_{(2)}~\hbox{and}~~ t\in \alpha^{-1}(e_{0});\label{(3.10)}
\end{equation}
\begin{equation}
(x\oplus z)\cdot y = x\cdot y \oplus z,~~~(\forall) (x,y)\in
G_{(2)}~\hbox{and}~~ z\in \beta^{-1}(e_{0}).\label{(3.11)}
\end{equation}
\end{Prop}
\begin{proof}
If $y\in\alpha^{-1}(e_{0})$, then $\alpha(y)=e_{0}$. We have
$\beta(\varepsilon(e_{0}))=e_{0}$, since $\beta\circ \varepsilon =
Id_{G_{0}}$. So $(\varepsilon(e_{0}),y)\in G_{(2)}$. Using (3.6)
and the condition $(G2)$ from Definition 2.1, one obtains $e\cdot
y = \varepsilon(e_{0})\cdot y = \varepsilon(\alpha(y))\cdot y =
y.~$ Hence the first relation of $(3.9)$ holds. Similarly, we
verify the second equality of $(3.9)$.

For to prove the relation $(3.10)$ we apply the interchange law
$(3.1)$ and $(3.9)$. Indeed, if in $(3.1)$ we replace $z$ with
$e$, one obtains $~(x\cdot y)\oplus (e\cdot t)= (x\oplus e)\cdot
(y\oplus t),$ for all $(x,y), (e,t)\in G_{(2)}.$ It follows
$(x\cdot y)\oplus t= x\cdot (y\oplus t),$ since $x\oplus e=x,
\beta(e)=e_{0}$ and $t\in \alpha^{-1}(e_{0}).$ Hence, the relation
$(3.10)$ holds.

Similarly, if in $(3.1)$ we replace $t$ with $e$, one obtains
$~(x\cdot y)\oplus (z\cdot e)= (x\oplus y)\cdot (y\oplus e),$ for
all $(x,y), (z,e)\in G_{(2)}.$ It follows $(x\cdot y)\oplus z=
(x\oplus z)\cdot y,$ since $y\oplus e= y, \alpha(e)=e_{0}$ and
$z\in \beta^{-1}(e_{0}).$ Hence, the relation $(3.11)$ holds.
\end{proof}
\begin{Prop}
{\rm (\cite{brmuc})} If $~(G, \alpha, \beta, m, \e, i, \oplus,
G_{0})$
is a ${\cal G}-$groupoid, then:\\[-0.2cm]
\begin{equation}
x\cdot y = x \oplus \overline{\e(\beta(x))} \oplus y, ~~~(\forall)
(x, y) \in G_{(2)};\label{(3.12)}\\[-0.1cm]
\end{equation}
\begin{equation}
x^{-1} = \e(\alpha(x))  \oplus \bar{x} \oplus \e(\beta(x)),
~~~(\forall) x \in G.\label{(3.13)}
\end{equation}
\end{Prop}
\begin{proof}
Fix  $x,y\in G$ and introduce the notations $u:=\alpha(x),
~v:=\beta(x) $ and $w:=\beta(y).$ Consider $(x,y)\in G_{(2)}$.
Then $\beta(x)=\alpha(y)=v.$ Since $ \oplus$ is associative, we
have\\[0.1cm]
$(1)~~~ x\cdot y =( (x\oplus \overline{\e(v)}~)\oplus \e(v)) \cdot
(e \oplus y).$

Using the fact that $\beta$ is a group morphism and
$(3.5)$ we have\\[0.1cm]
 $\beta (x \oplus
\overline{\e(v)} )=\beta(x)\oplus \beta( \overline{\e(v)} )= v
\oplus \overline{\beta (\e(v))}=v \oplus \bar{v}=e_{0}=\alpha(e).$
Then the product $(x \oplus \overline{\e(v)}~)\cdot e$ is defined
and $x \oplus \overline{\e(v)}\in \beta^{-1}(e_{0}).$ Applying
$(3.9)$ one obtains $(x \oplus \overline{\e(v)}~)\cdot e=x \oplus
\overline{\e(v)}.$
 On the other hand, we have $\e(v)\cdot y= \e(\alpha(y))\cdot y
 =y.$ Applying now the interchange law $(3.1),$ from $(1)$ implies
 that\\[0.1cm]
$x\cdot y =( (x\oplus \overline{\e(v} ) \cdot e)\oplus (
\e(v)\cdot y )~~~\Rightarrow~~~ x\cdot y = x \oplus
\overline{\e(v)}\oplus y ~~~\Rightarrow~~~x\cdot y = x \oplus
\overline{\e(\beta(x))}.$

For the prove $(3.13)$ denote $ a:= \e(\alpha(x)) \oplus \bar{x}
\oplus \e(\beta(x))$. Then $ a:= \e(u) \oplus \bar{x} \oplus
\e(v)$. Applying the fact that $\alpha$ is a group morphism and
$(3.5)$,  we have\\[0.1cm]
$\alpha(a) = u\oplus \overline{\alpha (x)} \oplus v = u \oplus
\bar{u}\oplus v= v.$ Then the product $x\cdot a $ is defined. We
have\\[0.1cm]
$(2)~~~ x\cdot a =( e\oplus x)\cdot ( (\e(u)\oplus \bar{x}) \oplus
\e(v)).$

Applying the interchange law $(3.1)$ and $(3.9)$, from $(2)$ we have\\[0.1cm]
$x\cdot a =( e\cdot (\e(u)\oplus \bar{x}))\oplus (x\cdot \e(v))
~~~\Rightarrow~~~ x\cdot a = \e(u)\oplus \bar{x} \oplus x
~~~\Rightarrow~~~x\cdot a = \e(u).$

Hence, $ x\cdot a = \e(\alpha(x)).$  Similarly, we verify that $
a\cdot x = \e(\beta(x)).$ Then $ a = x^{-1}$ and the relation
$(3.13)$ holds.
\end{proof}
\begin{Cor}
If $ (G, \alpha, \beta, m, \e, i, \oplus,  G_{0})$ is a ${\cal
G}-$groupoid, then:\\[-0.2cm]
\begin{equation}
x\cdot y = x \oplus y ~~~\hbox{and}~~~x^{-1} =
\bar{x},~~~~~(\forall) x, y \in G(e_{0}).\label{(3.14)}
\end{equation}
\end{Cor}
\begin{proof}
Let $x,y \in G(e_{0})$. We have $(x,y)\in G_{(2)}$, since
$\beta(x)=\alpha(y)= e_{0}$. From $(3.12)$, we have $x\cdot y
=x\oplus y, $ since $\e(\beta(x))= \e(e_{0})=e $ and $ \bar{e}=e$.
Hence, the first equality from $(3.14)$ holds. Replacing in
$(3.13)$, $\e(\alpha(x))= \e(\beta(x))=e$ one obtains $x^{-1} =
\bar{x}$. Therefore, the second equality from $(3.14)$ holds.
\end{proof}
\begin{Th}
Let $ (G, \alpha, \beta, m, \e, i,  G_{0})$ be a groupoid. If the
following conditions are satisfied:

$(i)~~(G, \oplus)$ and $ (G_{0}, \oplus)$ are groups;

$(ii)~ \alpha, \beta: (G, \oplus)\to (G_{0}, \oplus),~\e: (G_{0},
\oplus) \to (G, \oplus)~$ and $~i: (G, \oplus) \to (G, \oplus) $
are morphisms of groups;

 $(iii)~$ the interchange law $(3.1)$ between  the operations $ m $ and  $ \omega $
 holds,\\
 then $ (G, \alpha, \beta, m, \e, i, \oplus,  G_{0}) $
is a group-groupoid.
\end{Th}
\begin{proof}
By hypothesis, the condition $(i)$ from Definition 3.1 is
verified. It remains to prove that the condition $(ii)$ holds.

$(a)~$ We prove that $(\omega,\omega_{0}):(G\times G, G_{0}\times
G_{0}) \to (G, G_{0})$ is a  morphism of groupoids. Since $\alpha
$ is a morphism of groups, it follows $\alpha(x\oplus y)=
\alpha(x)\oplus \alpha(y)$, for all $x,y\in G.$ Then $\alpha
(\omega(x,y))=\omega_{0}(\alpha(x),\alpha(y)),$ and it follows
$\alpha (\omega(x,y))=\omega_{0}((\alpha \times \alpha )(x,y));$
i.e., $\alpha \circ \omega =\omega_{0}\circ (\alpha \times
\alpha).$ Similarly, we prove that $ \beta \circ \omega
=\omega_{0}\circ (\beta \times \beta).$  Hence the condition $(i)$
from Definition 2.3 is satisfied.

We suppose that the interchange law $(3.1)$ holds. Then,
 for all $(x,y)$ and $(z,t)$ in $G_{(2)}$ we
 have $~(x\cdot y)\oplus (z\cdot t)= (x\oplus z)\cdot (y\oplus t).$
From the last equality it follows\\[0.1cm]
$m(x,y)\oplus m(z,t)=\omega(x,z)\cdot \omega(y,t)~~\Rightarrow ~~
\omega(m(x,y),m(z,t))= m(\omega(x,z),(\omega(y,t)).$

Then $\omega(m_{G\times G}((x,y),(z,t)))=
m(\omega(x,z),(\omega(y,t)),$ and the condition $(ii)$ from
Definition 2.3 holds. Hence, $ (\omega, \omega_{0})$ is a groupoid
morphism.

$(b)~$ We prove that $~(\nu, \nu_{0})~$ is a morphism of
  groupoids (here $\{\lambda\}$ is regarded as null groupoid with
  the structure functions
  $\alpha_{0}^{\prime}, \beta_{0}^{\prime}, \e_{0}^{\prime}, i_{0}^{\prime} $ and
  multiplication $m_{0}^{\prime} ).$ Since  $\alpha$ and $\e$ are group morphisms, we have
  $\alpha(e)=e_{0}$ and $\e(e_{0})=e$. From $\alpha(\nu(\lambda))=\alpha(e)=e_{0}$ and
$\nu_{0}(\lambda) = e_{0}$, it follows $\alpha \circ \nu =
\nu_{0}\circ Id.$ Similarly, we have $ \beta\circ \nu =
\nu_{0}\circ Id.$  Also, we have
 $~\nu (m_{0}^{\prime}(\lambda,
\lambda))= \nu (\lambda\cdot \lambda)= \nu (\lambda)=e $
and\\[0.1cm]
$m(\nu(\lambda), \nu(\lambda))=e\cdot e= \e(e_{0})\cdot e =
\e(\alpha (e))\cdot e = e.$ Then, $\nu (m_{0}^{\prime}(\lambda,
\lambda))=m(\nu(\lambda), \nu(\lambda)).$ Hence, the pair $(\nu,
\nu_{0})$ is a groupoid morphism.

 $(c)~$ We prove that $~(\sigma, \sigma_{0})~$ is a groupoid morphism.
 Applying $(3.6)$ we have $\alpha (
\sigma(x))=\alpha(\bar{x})=\overline{\alpha(x)} $  and
$\sigma_{0}(\alpha(x))= \overline{\alpha(x)}.$ Then $ \alpha\circ
\sigma = \sigma_{0}\circ \alpha.$ Similarly, we have $ \beta\circ
\sigma = \sigma_{0}\circ \beta.$ We shall prove that:\\[0.1cm]
 $(1)~~~\overline{x\cdot y} = \bar{x}\cdot
 \bar{y},~~(\forall)~(x,y)\in G_{(2)}.$

From $(x,y)\in G_{(2)}$ we have $\beta(x)=\alpha(y)$. Then
$\overline{\beta(x)}=\overline{\alpha(y)}$, and it follows
$\beta(\bar{x})=\alpha(\bar{x})$. Therefore $(\bar{x},\bar{y})\in
G_{(2)}$. Using now $(3.1)$ one obtains\\[0.1cm]
$(2)~~~(x\cdot y)\oplus (\bar{x}\cdot \bar{y})= (x \oplus
\bar{x})\cdot (y \oplus \bar{y})~~~\hbox{and}~~~ (\bar{x}\cdot
\bar{y})\oplus (x\cdot y) = (\bar{x} \oplus x)\cdot (\bar{y}
\oplus y).$

Using the relations  $a\oplus \bar{a}=\bar{a}\oplus a= e,~$ and $~e\cdot e=e, $ from $(2)$, we have\\[0.1cm]
$(3)~~~(x\cdot y)\oplus (\bar{x}\cdot \bar{y})= e
~~~\hbox{and}~~~(\bar{x}\cdot \bar{y})\oplus (x\cdot y)= e. $

From $(3)$ one obtains that the equality $(1)$ holds.

The relation $(1)$ is equivalently with\\[0.1cm]
$\sigma(x\cdot y)= \sigma(x)\cdot \sigma(y)~~\Leftrightarrow ~
\sigma(m_{G}(x,y)) = m_{G}(\sigma(x),\sigma(y)).$

Hence, $ (\sigma, \sigma_{0})$ is a groupoid morphism.
\end{proof}
According to Proposition  $3.1$ and Theorem $3.1$ we give a new
definition (Definition $3.2$) for the notion of group-groupoid
(this is equivalent with Definition $3.1$).
\begin{defn}
{\rm A {\it group-groupoid} is a groupoid $ (G,\alpha, \beta, m,
\e, i, G_{0})$ such that the following conditions are satisfied:

$(i)~~ (G, \oplus)$ and $ (G_{0}, \oplus )$ are groups;

$(ii)~~ \alpha, \beta: (G, \oplus) \to (G_{0}, \oplus),~ \e:
(G_{0}, \oplus) \to (G, \oplus)~$ and $~i: (G, \oplus) \to (G,
\oplus) $ are morphisms of groups;

$(iii)~ $ the interchange law $(3.1)$ between  the operations $m$
and  $\oplus$ holds.} \hfill$\Box$
\end{defn}

If in Definition 3.2, we consider $ G_{0}\subseteq G$ and $
\varepsilon: G_{0}\to G$ is the inclusion map, then $ (G, \alpha,
\beta, m, i, \oplus, G_{0})$ is a group-groupoid and we  say that
it is a {\it group$-G_{0}-$groupoid}.
\begin{Ex}
{\rm $(i)~$ Let $G_{0}$ be a group. Then $ G_{0}$ has a structure
of null groupoid over $G_{0}$ (see Example 2.1(i)). We have that
$G= G_{0}$ and $ \alpha, \beta,\varepsilon, i $ are morphisms of
groups. It is easy to prove that the interchange law $(3.1)$ is
verified. Then $G_{0}$ is a group-groupoid, called the {\it null
group-groupoid} associated to group $G_{0}$.

$(ii)~$ A commutative  group $(G, \oplus)$ having $\{e\}$ as unity
may be considered to be a $\{ e\}-$groupoid (see Example 2.1(ii)).
In this case, $ m =\oplus$. We have that $(G, \oplus)$ and
$G_{0}=\{e\}$ are groups. It is easy to see that $\alpha, \beta,
\e $ and $i$ are morphisms of groups. It remains to verify that $
(3.1)$ holds. Indeed, for $x,y,z,t \in G$ we have\\
 $(x\oplus y)\oplus (z\oplus t)= (x\oplus z )\oplus (y \oplus t),$ since the
operation $\oplus$ is associative and commutative. Hence $(G,
\alpha, \beta, m, \e, i, \oplus, \{e\})$ is a group-groupoid
called {\it group-groupoid with a single unit} associated to
commutative group $(G, \oplus)$.
 Therefore, {\it each commutative group $ G $ can be regarded as a commutative group
 $-\{e\}-$groupoid}.} \hfill$\b$
\end{Ex}
\begin{Ex}
 {\rm Let $(G, \oplus)$ be a group and  $( G\times G,
\overline{\alpha}, \overline{\beta},
\overline{m},\overline{\varepsilon}, \overline{i}, G )$ the pair
groupoid associated to $G$ (see Example 2.1(iii)). We have that $
G\times G$ is a group endowed with operation  $ (x_{1}, x_{2})
\oplus (y_{1}, y_{2}):=( x_{1}\oplus y_{1}, x_{2}\oplus y_{2}),~$
for all $ x_{1}, x_{2}, y_{1}, y_{2}\in G$. It is easy to check
that $ \overline{\alpha}, \overline{\beta}, \overline{\varepsilon}
 $ and $ \overline{i} $ are group morphisms. Therefore, the
conditions $(i)$ and $(ii)$ from Definition $3.2$ are satisfied.
 For to prove that the condition $(iii)$ is verified, we consider $
x=(x_{1}, x_{2}), y=(y_{1}, y_{2}), z=(z_{1}, z_{2}), t=(t_{1},
t_{2}) $ from $G\times G$ such that
$\overline{\beta}(x)=\overline{\alpha}(y) $ and
$\overline{\beta}(z)= \overline{\alpha}(t).$ Then $ x_{2}=y_{1} $
and $ z_{2}=t_{1}$. It follows $y=(x_{2}, y_{2}),~t=(z_{2},
t_{2}), x\cdot y = (x_{1}, y_{2}) $ and $  z\cdot t = (z_{1},
t_{2}). $ We have\\
$(x\cdot y)\oplus (z\cdot t) = (x_{1}, y_{2})\oplus (z_{1},
t_{2})= (x_{1}\oplus z_{1}, y_{2}\oplus t_{2})~$ and\\[0.1cm]
$ (x\oplus z)\cdot (y\oplus t)=  (x_{1}\oplus z_{1}, x_{2}\oplus
z_{2})\cdot (x_{2}\oplus z_{2}, y_{2}\oplus t_{2})=(x_{1}\oplus
z_{1}, y_{2}\oplus t_{2}).~$ Then, $ (x\cdot y)\oplus (z\cdot
t)=(x\oplus z)\cdot (y\oplus t)$ and so  $(3.1)$ holds. Hence
$G\times G$ is a group-groupoid called the  {\it group-pair
groupoid} associated to group $G$.}
 \hfill$\b$
\end{Ex}
\begin{Ex}
 {\rm Consider the groups  $({\bf R}^{2}, + )$ and $({\bf R}, +).$
For $(G, G_{0})$ where $G:={\bf R}^{2}$ and $G_{0}:={\bf R}$, we
define the structure functions $\alpha, \beta: {\bf R}^{2} \to
{\bf R},~\varepsilon :{\bf R}\to {\bf R}^{2}~$ and $~i:{\bf
R}^{2}\to {\bf R}^{2}$ as follows: $~ \alpha(x_{1},x_{2}):= x_{1}+
2 x_{2},~~\beta(x_{1},x_{2}):= x_{1}+ x_{2},$\\[0.1cm]
$\varepsilon(x_{1}):=(x_{1},0)~$ and $~i(x_{1},x_{2}):=(x_{1}+ 3
x_{2}, -x_{2}), $ for all $x_{1}, x_{2}\in {\bf R}.$

Let $ G_{(2)}:=\{((x_{1}, x_{2}),(y_{1}, y_{2} ))\in {\bf
R}^{2}\times {\bf R}^{2}~|~ x_{2}= - x_{1} + y_{1} + 2 y_{2}\}$ be
the set of composable pairs. The multiplication $ m:
G_{(2)} \to G$ is given by:\\[0.1cm]
$ (x_{1}, x_{2})\cdot (y_{1}, y_{2}): = (x_{1}- 2 y_{2}, x_{2} +
y_{2}),~~~~~\hbox{if}~~~x_{2}= - x_{1} + y_{1} + 2 y_{2}.$

It is easy to check that the above structure functions determine
on $G$ a structure of a groupoid over $G_{0}$. Also, the maps
$\alpha, \beta, \varepsilon$ and $i$ are group morphisms.
Therefore, the conditions $(i)$ and $(ii)$ from the Definition 3.2
hold.

We consider $x,y, z,t \in {\bf R}^{2}$ such that the products
$x\cdot y$ and $z\cdot t$ are defined. Then $ x=(x_{1}, x_{2}),
y=(y_{1}, y_{2}),  z=(z_{1}, z_{2}), t=(t_{1}, t_{2})$  such that
$ x_{2}= - x_{1} + y_{1} + 2 y_{2}$ and $z_{2}= - z_{1} + t_{1} +
2 t_{2}.$ We have $~ x\cdot y =(x_{1}- 2 y_{2}, x_{2} + y_{2}),~
z\cdot t =(z_{1}- 2 t_{2}, z_{2} + t_{2}).$ Then $~(x\cdot y) +
(z\cdot t) = (x_{1}- 2 y_{2}+ z_{1}- 2 t_{2},x_{2} +
y_{2}+z_{2} + t_{2})~$ and\\[0.1cm]
$(x + z)\cdot (y+t) = (x_{1}+ z_{1}-2(y_{2}+ t_{2}), x_{2}+ z_{2}+
y_{2}+ t_{2}).$

Hence, $(x\cdot y) + (z\cdot t) = (x + y)\cdot (z+t)$ and the
interchange law $(3.1)$ holds. Therefore, $({\bf R}^{2}, \alpha,
\beta, m, \varepsilon, i, {\bf R}) $ is a commutative group-
groupoid.

Let us we consider the Euclidean plane ${\bf R}^{2}$ with the
Cartesian coordinate system $ O x_{1} x_{2}.$ The $\alpha-$fibres
$\alpha^{-1}(u)$ for $u\in {\bf R}$ are represented by parallel
straight lines of equation $x_{1} + 2 x_{2} - u = 0.$ Also, the
$\beta-$fibres $\beta^{-1}(v)$ for $v\in {\bf R}$ are represented
 by parallel straight lines of equation $x_{1} + x_{2}
- v = 0.$

 Let be the points $A_{1}, A_{2}, A_{3}, A_{4}$ associated to
elements $\varepsilon(\beta(x)), x, x\cdot y, y\in G$, for
$\beta(x)=\alpha(y).$ Then $A_{1}(b+2c, 0), A_{2}(a, -a+b+2c),
A_{3}(a-2c, -a+b+3c) $ and $ A_{4}(b, c). $ We have that: {\it the
simple quadrilateral $A_{1}A_{2}A_{3}A_{4}$ is a parallelogram}.

Indeed, the slope of line through $A_{1}$ and $A_{4}$ is
$m_{A_{1}A_{4}}=-\frac{1}{2}$ and the distance  from $A_{1}$ and
$A_{4}$ is $ d(A_{1}, A_{4})=|c|\sqrt{5}.$ Also,
$m_{A_{2}A_{3}}=-\frac{1}{2}$ and $ d(A_{2}, A_{3})=|c|\sqrt{5}.$

Let be the points $B_{1}, B_{2}, B_{3}, B_{4}$ associated to
$\varepsilon(\alpha(x)), x, \varepsilon(\beta(x)), x^{-1}\in G.$
We have $ B_{1}(x_{1}+2x_{2}, 0), B_{2}(x_{1}, x_{2}),
B_{3}(x_{1}+x_{2}, 0) $ and $ B_{4}(x_{1}+3x_{2}, -x_{2}). $ Then:
 {\it the simple quadrilateral $B_{1}B_{2}B_{3}B_{4}$ is a
parallelogram}.

Indeed, we have $m_{B_{1}B_{2}}=m_{B_{3}B_{4}}=-\frac{1}{2}$ and $
d(B_{1}, B_{2})= d(B_{3}, B_{4})= |x_{2}|\sqrt{5}.$}\hfill$\Box$
\end{Ex}
\begin{defn}
{\rm Let $(G ,\alpha, \beta, m, \e, i, \oplus, G_0)$  be a
group-groupoid. A subgroupoid $(H,H_0)$ of the groupoid $(G,
G_{0}) $ is called a \textit{group-subgroupoid} of $(G, G_{0}) $,
if $ H $  and $ H_{0} $ are subgroups in $G$ and $G_{0}$,
respectively.

If $H_{0}=G_{0}$ we say that $H$ is a \textit{group
$-G_{0}-$subgroupoid} of $(G, G_{0})$}.\hfill$\Box$
\end{defn}
 According to Definition 3.3, if $(H,H_0)$ is a group-subgroupoid of
 $(G, G_0)$, then the pair $(H,H_0)$ endowed with the restrictions of
functions $\alpha, \beta , i ~\hbox{and}~ \oplus $ to $H$,
$\varepsilon $ to $H_{0}$ and  $ m $ to $ H_{(2)}$, is a
group-groupoid, denoted by $(H,H_0).$ We denote by the same
letters the structure functions of $H$ as well as those of $G$.
\begin{Prop}
Let $(G ,\alpha, \beta, m, \e, i, \oplus, G_0)$ be a
group-groupoid. Then:

$(i)~$ The  fibres $\alpha^{-1}(e_{0})$ and $\beta^{-1}(e_{0})$
are subgroups in $G$.

$(ii)~$ The isotropy group $G(e_{0})$ is a
group$-\{e_{0}\}-$subgroupoid of $G$.

$(iii)~~\varepsilon(G_{0})$  and $ G$ are group
$-G_{0}-$subgroupoids of $G$.

$(iv)~~Is(G):=\{ x\in G~|~\alpha (x)=\beta (x)\}~$ is a group
$-G_{0}-$subgroupoid of $G$.
\end{Prop}
{\it Proof.} $(i)~$ For all $x,y\in \alpha^{-1}(e_{0})$ we have
$x\oplus y\in \alpha^{-1}(e_{0})$ and $\bar{x}\in
\alpha^{-1}(e_{0})$. Indeed, applying $(3.2)$ we have
$\alpha(x\oplus y)=\alpha(x)\oplus \alpha(y)= e_{0}\oplus
e_{0}=e_{0}$. Also, using $(3.5)$ it follows
$\alpha(\bar{x})=\overline{\alpha(x)}=\overline{e_{0}}=e_{0}$.
Then $\alpha^{-1}(e_{0})$ is a subgroup in $(G,\oplus)$.
Similarly, we prove that $\beta^{-1}(e_{0})$ is a subgroup in
$(G,\oplus)$.

$(ii)~$ It is easy to verify that $G(e_{0})$ is a
$\{e_{0}\}-$subgroupoid. Also, according to $(i)$ we have that
$G(e_{0})$ is a subgroup of $(G,\oplus)$, since $G(e_{0})=
\alpha^{-1}(e_{0})\cap \beta^{-1}(e_{0})$. Hence, the conditions
from Definition $3.3$ are satisfied and $G(e_{0})$ is a
group$-\{e_{0}\}-$subgroupoid.

$(iii)~$ It is easy to verify that $ \varepsilon(G_{0})$ and $G$
are group $-G_{0}-$subgroupoids.

$(iv)~$  Clearly, $\alpha(Is(G))=\beta(Is(G))=G_{0}.$ Let $x,y\in
Is(G)$ with $ (x,y)\in G_{(2)}$. Then
$\alpha(x)=\beta(x)=\alpha(y)=\beta(y)$. We have
$\alpha(xy)=\alpha(x)=\beta(y)=\beta(xy)$ and
$\alpha(x^{-1})=\beta(x)=\alpha(x)=\beta(x^{-1})$. Hence,  $ xy,
x^{-1}\in Is(G)$ and $Is(G)$ is a sugroupoid. Since $\alpha$ and
$\beta$ are group morphisms, implies that $~\alpha( x \oplus y
)=\beta (x \oplus  y)~$ and $~\alpha(\bar{x})=\beta(\bar{x}). $
Then, for all $x,y \in Is(G)$ we have  $ x \oplus y, \bar{x} \in
Is(G).$  Therefore, $ Is(G)$ is a subgroup in $(G,\oplus)$. Hence,
$Is(G)$ is a group $-G_{0}-$subgroupoid. \hfill$\Box$

 The group-subgroupoid $ Is(G)$ is the
union of all isotropy groups of $G$ and it is called the {\it
isotropy bundle} of the group-groupoid $(G,G_{0})$.
\begin{defn}
{\rm Let $ ( G_{j}, \alpha_{j}, \beta_{j}, m_{j}, \varepsilon_{j},
i_{j}, \oplus_{j}, G_{j,0} ),~ j=1,2$ be two group-groupoids. A
groupoid morphism  $ (f, f_{0}): (G_{1}, G_{1,0})\to (G_{2},
G_{2,0}) $ such  that $ f $ and $ f_{0}$ are group morphisms, is
called  {\it group-groupoid morphism} or {\it morphism of
group-groupoids}.

A group-groupoid morphism of the form $ (f, Id_{G_{1,0}}): (G_{1},
G_{1,0})\to (G_{2}, G_{1,0}) $ is called  {\it $ G_{1,0}-$morphism
of group-groupoids}.  It is denoted by $ f: G_{1} \to
G_{2}$.}\hfill$\Box$
\end{defn}

The category of group-groupoids, denoted by ${\mathcal G}Gpd$, has
its objects all group-groupoids $(G, G_{0})$ and as morphisms from
$(G, G_{0})$ to $(G^{\prime}, G_{0}^{\prime})$  the set of all
morphisms of group-groupoids.
\begin{Ex}
{\rm {\bf Direct product of two group-groupoids}. Let given the
group-groupoids $ (G, G_{0}) $ and $ (K, K_{0}) $. We consider the
direct product  $ (G\times K, G_{0}\times K_{0})~$ of the
groupoids $(G, G_{0}) $ and $ (K, K_{0})$ (see Example 2.1 (iv)).
On $ G\times K$ and $ G_{0}\times K_{0}$ we
introduce the usual group operations. These operations are defined by\\
$(g_{1}, k_{1})\oplus_{G\times K} (g_{2}, k_{2}):= ( g_{1}
\oplus_{G} g_{2}, k_{1} \oplus_{K} k_{2} ), ~ (\forall) g_{1}, g_{2}\in G, k_{1}, k_{2}\in K$  ~~\hbox{and}\\
$(u_{1}, v_{1})\oplus_{G_{0}\times K_{0}} (u_{2}, v_{2}):= ( u_{1}
\oplus_{G_{0}} u_{2}, v_{1} \oplus_{K_{0}} v_{2} ),~ (\forall)
u_{1}, u_{2}\in G_{0}, v_{1}, v_{2}\in K_{0}$.

By a direct computation we prove that the conditions from
Definition 3.2 are satisfied. Then  $ (G\times K, G_{0}\times
K_{0})$ is a group-groupoid, called the {\it direct product of
group- groupoids} $(G,G_{0})$ and $(K,K_{0})$. The canonical
projections $ pr_{G} : G\times K \to G$ and $ pr_{K} : G\times K
\to K$ are morphisms of group-groupoids.}\hfill$\Box$
\end{Ex}
\begin{Prop}
Let $(G,\alpha, \beta, m, \varepsilon, i , \oplus, G_{0})$ be a
group-groupoid. The anchor map $(\alpha, \beta): G \to G_{0}\times
G_{0} $ is a $ G_{0}-$ morphism of group-groupoids from the
group-groupoid $(G, G_{0})$  into the group-pair groupoid $ (
G_{0}\times G_{0}, \overline{\alpha}, \overline{\beta},
\overline{m}, \overline{\varepsilon}, \overline{i}, \oplus, G_{0})
$.
\end{Prop}
\begin{proof}
 We denote $ (\alpha, \beta):=f $. Then $ f(x)=(\alpha(x),\beta(x)),$ for all $x\in G.$ We prove that
 $ \overline{\alpha}\circ f = \alpha $. Indeed, for all $x\in G$ we
have $ (\overline{\alpha}\circ f)(x)=\overline{\alpha}(\alpha(x),
\beta(x)) =\alpha(x).$
 Therefore, $\overline{\alpha}\circ f = \alpha $. Similarly, we
verify that $ \overline{\beta}\circ f = \beta $.

 For $(x,y)\in G_{(2)}$ we have
 $ f(x\cdot y )= (\alpha(x\cdot y),
\beta(x\cdot y))= (\alpha(x), \beta(y))~$ and\\[0.1cm]
 $~\overline{m}(f(x),f(y))=
\overline{m}((\alpha(x),\beta(x)), (\alpha(y),\beta(y))) =
(\alpha(x),\beta(y)),$ since $\beta(x)=\alpha(y).$ Therefore, $ f(
x\cdot y )=\overline{m}(f(x),f(y))$. Hence, $ f $ is a  $G_{0}-$
morphism of groupoids.

Let $ x, y \in G .$ Since $ \alpha, \beta $ are group morphisms,
we have\\[0.1cm]
$f(x\oplus y)=(\alpha(x) \oplus \alpha(y), \beta(x) \oplus
\beta(y))= f(x)\oplus f(y) $, i.e. $f$ is a morphism of groups.
Hence $ f $ is a $ G_{0}-$ morphism of group-groupoids.
\end{proof}

\vspace*{0.2cm}

\hspace*{0.7cm}West University of Timi\c soara\\
\hspace*{0.7cm} Department for Training of Teachers (D.P.P.D.)\\
\hspace*{0.7cm} Bd. V. P{\^a}rvan,no.4, 300223, Timi\c soara, Romania\\
\hspace*{0.7cm}E-mail: ivan@math.uvt.ro

\end{document}